\documentclass[11pt,a4paper,reqno]{amsart}
\usepackage{latexsym}
\usepackage{bbm,amssymb,amsthm,amsmath}
\usepackage{graphicx}
\usepackage{multirow}
\usepackage{caption}
\usepackage{bm}
\usepackage{multirow}
\usepackage{mathtools}  
\usepackage{float}
\usepackage{xcolor}

\usepackage{xcolor}

\calclayout

\makeatletter	
\@namedef{subjclassname@2020}{%
	\textup{2020} Mathematics Subject Classification}
\makeatother

\theoremstyle{plain}
\newtheorem{theorem}{Theorem}[section]	
\newtheorem{lemma}[theorem]{Lemma}
\newtheorem{corollary}[theorem]{Corollary}

\theoremstyle{definition}
\newtheorem{definition}[theorem]{Definition}
\newtheorem{example}[theorem]{Example}

\theoremstyle{remark}
\newtheorem{remark}[theorem]{Remark}

\numberwithin{equation}{section}	

\def\C{\mathbb{C}}
\def\R{\mathbb{R}}    
  
\def\N{\mathbb{N}}

\def\vu{\mathsf{u}}

\def\mA{\mathbf{A}}

\def\lD{\mathcal{D}}

\def\lG{\mathcal{G}}
\def\lH{\mathcal{H}}
\def\lL{\mathcal{L}}
\def\lV{\mathcal{V}}
\def\lW{\mathcal{W}}

\def\dom{\operatorname{dom}}
\def\dim{\operatorname{dim}}
\def\ker{\operatorname{ker}}
\def\ran{\operatorname{ran}}

\def\scp#1#2{\langle\, #1 \mid #2 \,\rangle}  
\def\iscp#1#2{[\, #1 \mid #2 \,]}  
\def\iscpA#1#2{[\, #1 \mid #2 \,]_A}  

\def\mA{{\bold A}}

\def\vu{{\sf u}}

\begin{document}

\title[The von Neumann extension theory for abstract Friedrichs operators]{The von Neumann extension theory for abstract Friedrichs operators}

\author{M.~Erceg}\address{Marko Erceg,
	Department of Mathematics, Faculty of Science, University of Zagreb, Bijeni\v{c}ka cesta 30,
	10000 Zagreb, Croatia}\email{maerceg@math.hr}

\author{S.~K.~Soni}\address{Sandeep Kumar Soni,
	Department of Mathematics, Faculty of Science, University of Zagreb, Bijeni\v{c}ka cesta 30,
	10000 Zagreb, Croatia}\email{sandeep@math.hr}




\begin{abstract}
The theory of abstract Friedrichs operators was introduced some fifteen years ago with the aim of providing a more comprehensive framework for the study of positive symmetric systems of first-order partial differential equations, nowadays better known as (classical) Friedrichs systems.
Since then, the theory has not only been frequently applied in numerical and analytical research of Friedrichs systems, but has continued to evolve as well. 
In this paper we provide an explicit characterisation and a classification of abstract Friedrichs operators.
More precisely, we show that every abstract Friedrichs operators can be written as the sum of a skew-selfadjoint operator and a bounded self-adjoint strictly-positive operator. 
Furthermore, we develop a classification of realisations of abstract Friedrichs operators in the spirit of the von Neumann extension theory, which, when applied to the symmetric case, extends the classical theory. 
\end{abstract}

\maketitle

%

\section{Introduction}\label{sec:intro}

\emph{Abstract Friedrichs operators} were introduced by Ern, Guermond and Caplain \cite{EGC} some fifteen years ago with the aim of providing a more comprehensive framework for the study of \emph{positive symmetric systems}.
These systems originate from the work of Friedrichs \cite{KOF} (following his research on symmetric hyperbolic systems \cite{KOFh}), and today are customarily referred to as \emph{(classical) Friedrichs systems}.
The reason why these systems are still attractive to the community lies in the fact that a wide variety of (semi)linear equations of mathematical physics (regardless of their order), including classical elliptic, parabolic and hyperbolic equations, can be adapted, or rewritten, in the form of Friedrichs systems.
Moreover, the same applies to equations that change their type (the so-called mixed-type equations), such as the Tricomi equation, the study of which was actually the main motivation of Friedrichs to introduce this concept.
A nice historical exposition of the classical Friedrichs' theory (which was very active until 1970's) can be found in \cite{MJensen}.

The renewed interest in Friedrichs systems arose from numerical analysis (see e.g.~\cite{HMSW, MJensen}) based on the need to apply (discontinuous) Galerkin finite element methods to partial differential equations of various types.
The well-posedness results for Friedrichs systems obtained within the classical theory were not satisfactory; there were only results on the existence of weak solutions, and the uniqueness of strong ones, leaving the general question open on the joint existence and uniqueness of either a weak or a strong solution.
This brings us to abstract Friedrichs operators since in \cite{EGC} the authors obtained a proper well-posedness result (Theorem \ref{thm:abstractFO-prop}(xii) below) within the operator-theoretic framework introduced for this purpose (see also \cite{ABcpde}).
This novel approach initiated a number of new investigations in various directions. For example, studies of different representations of boundary conditions and the relation with classical theory \cite{ABcpde, ABjde, ABVisrn, AEM-2017, BH21, ES22}, applications to diverse (initial-)boun\-dary value problems of elliptic, hyperbolic, and parabolic type \cite{ABVjmaa, BEmjom, BVcpaa, EM19, EGsemel, MDS}, and the development of different numerical schemes \cite{BDG, BEF, CM21, CHWY23, EGbis, EGter}.

While we postpone the introduction of the precise definition of abstract Friedrichs operators to the subsequent section (Definition \ref{def:abstractFO}), here we discuss the main ideas. 
Assume we are given two densely defined linear operators $T_0$, $\widetilde{T}_0$ on a Hilbert space $\lH$ such that $T_0\subseteq \widetilde{T}_0^*$ and $\widetilde{T}_0\subseteq T_0^*$ ($T_0^*$ denotes the adjoint operator of $T_0$ and $\widetilde{T}_0\subseteq T_0^*$ is understood in the standard way: $\dom \widetilde{T}_0\subseteq\dom T_0^*$ and $T_0^*|_{\dom\widetilde{T}_0}=\widetilde{T}_0$). Now we seek for realisations (or extensions) $T$ of $T_0$, i.e.~$T_0\subseteq T\subseteq \widetilde{T}_0^*$, such that $T:\dom T\to\lH$ is bijective. Namely, for any right-hand side $f\in\lH$ the inverse linear problem induced by $T$:
$$
Tu=f \,,
$$
has a unique solution $u\in\dom T$. Hence, if we think of $T$ as a differential operator (e.g.~a classical Friedrichs operator), then the associated problem is well-posed and the choice of realisation $T$ corresponds to the prescribed (initial-)boundary conditions. 
Then it is natural to develop a mechanism for recognition, i.e.~classification, of all bijective realisations.
Furthermore, we distinguish bijective realisations with signed boundary map (see Theorem \ref{thm:abstractFO-prop}(xii) below; cf.~\cite{ABcpde, EGC}), and a particular representative of such a class is given by $\dom T=\dom T^*$.
The latter case is particularly desirable in the analysis, as it is evident in \cite[Section 4]{EGC}, where such a situation ensures existence of certain projectors which were needed to relate different abstract notions of imposing boundary conditions (see also \cite{ABcpde}). On the other hand, in \cite[Theorem 18]{AEM-2017} the same is required to have a complete characterisation of all bijective realisations with signed boundary map.

Concerning the classification of all bijective realisations, in \cite{AEM-2017} Grubb's universal operator extension theory for the non-symmetric setting \cite{Grubb-1968} (see also \cite[Chapter 13]{Grubb}) was applied. 
The first main result of this contribution is the development of an alternative approach based on the adaptation (or generalisation) of the von Neumann extension theory for symmetric operators (cf.~\cite[Section 13.2]{Schmudgen}), which in a way has already started in \cite[Section 3]{ES22} by deriving the decomposition of the graph space (see Theorem \ref{thm:abstractFO-prop}(ix) below).  
This new approach proved to be better when restricting only to bijective realisations with signed boundary map (see Theorem \ref{thm:classif} below). Moreover, it is recognised that bijective realisations $T$ with the property that $\dom T=\dom T^*$, mentioned above, exist if and only if the kernels
of maximal operators $\widetilde{T}_0^*$ and $T_0^*$ are isomorphic.

In the second main result we show that every abstract Friedrichs operator can be written as a sum of a skew-selfadjoint operator and a bounded self-adjoint strictly-positive operator (see Theorem \ref{thm:abstractFO_L0+S} below). 
This result enables a convenient connection between theories of abstract Friedrichs and (skew-)symmetric operators, many aspects of which are addressed throughout the manuscript. In particular in the last section, where a generalisation of the von Neumann extension theory for symmetric operators is provided. 
Therefore, we hope that this article will bring the theory of abstract Friedrichs systems closer to a wider group of people working (or interested) in operator theory.

The paper is organised as follows. 
In Section \ref{sec:abstractFO} we recall the definition of abstract Friedrichs systems and concisely present the main properties. Two main results of the paper are the subject of the following two sections.
A characterisation of abstract Friedrichs operators is provided in Section \ref{sec:charact}, while the classification in the spirit of the von Neumann extension theory is developd in Section \ref{sec:classif}. Finally, the paper is closed with some corollaries for the symmetric setting.

\section{Abstract Friedrichs operators}\label{sec:abstractFO}

\subsection{Definition and main properties}

The abstract Hilbert space formalism for Friedrichs systems which we 
study in this paper was introduced 
and developed in \cite{EGC, ABcpde} for real vector spaces, while 
the required differences for complex vector spaces have been 
supplemented more recently in \cite{ABCE}. 
Here we present the definition in the form given in
\cite[Definition 1]{AEM-2017}.

\begin{definition}\label{def:abstractFO}
A (densely defined) linear operator $T_0$ on a complex Hilbert space $\lH$
(a scalar product is denoted by
$\scp\cdot\cdot$, which we take to be anti-linear in the second entry)
is called an \emph{abstract Friedrichs operator} if there exists a
(densely defined) linear operator $\widetilde{T}_0$ on $\lH$ with the following properties:
\begin{itemize}
 \item[(T1)] $T_0$ and $\widetilde{T}_0$ have a common domain $\lD$, i.e.~$\dom T_0=\dom\widetilde{T}_0=\lD$,
 which is dense in $\lH$, satisfying
 \[
 \scp{T_0\phi}\psi \;=\; \scp\phi{\widetilde T_0\psi} \;, \qquad \phi,\psi\in\mathcal{D} \,;
 \]
 \item[(T2)] there is a constant $\lambda>0$ for which
 \[
 \|(T_0+\widetilde{T}_0)\phi\| \;\leqslant\; 2\lambda\|\phi\| \;, \qquad \phi\in\mathcal{D} \,;
 \]
 \item[(T3)] there exists a constant $\mu>0$ such that
 \[
 \scp{(T_0+\widetilde{T}_0)\phi}\phi \;\geqslant\; 2\mu \|\phi\|^2 \;, \qquad \phi\in\mathcal{D} \,.
 \]
\end{itemize}
The pair $(T_0,\widetilde{T}_0)$ is referred to as a \emph{joint pair of abstract Friedrichs operators}
(the definition is indeed symmetric in $T_0$ and $\widetilde{T}_0$).
\end{definition}

Before moving to the main topic of the paper, 
let us briefly recall the essential properties 
of (joint pairs of) abstract Friedrichs operators, 
which we summarise in the form of a theorem.
At the same time, we introduce the notation that is used throughout the paper.
The presentation consists of two steps: first we deal with the consequences 
of conditions (T1)--(T2), and then we highlight the additional structure implied by condition (T3). 
A similar approach can be found in \cite[Theorem 2.2]{BEW23}.
This result enables a convenient connection between theories of abstract Friedrichs and symmetric operators, many aspects of which are addressed throughout the manuscript.
\begin{theorem}\label{thm:abstractFO-prop}
Let a pair of linear operators $(T_0,\widetilde{T}_0)$ on $\lH$ satisfy {\rm (T1)} and {\rm (T2)}. Then the following holds.
\begin{enumerate}
\item[i)] $T_0\subseteq \widetilde{T}_0^*=:T_1$ and $\widetilde{T}_0\subseteq T_0^*=:\widetilde{T}_1$, where 
$\widetilde{T}_0^*$ and $T_0^*$ are adjoints of $\widetilde{T}_0$ and $T_0$, respectively.

\item[ii)] The pair of closures $(\overline{T}_0,\overline{\widetilde{T}}_0)$ satisfies {\rm (T1)--(T2)} with the same constant $\lambda$.

\item[iii)] $\dom \overline{T}_0=\dom\overline{\widetilde{T}}_0=:\lW_0$ and $\dom T_1=\dom\widetilde{T}_1=:\lW$.

\item[iv)] The graph norms $\|\cdot\|_{T_1}:=\|\cdot\|+\|T_1\cdot\|$ and $\|\cdot\|_{\widetilde T_1}
	:=\|\cdot\|+\|\widetilde T_1\cdot\|$ are equivalent, $(\lW,\|\,\cdot\,\|_{T_1})$ is a Hilbert space (the \emph{graph space})
	and $\lW_0$ is a closed subspace containing $\lD$.
	
\item[v)] The linear operator $\overline{T_0+\widetilde{T}_0}$ is everywhere defined, bounded and self-adjoint on 
$\lH$ such that on $\lW$ it coincides with $T_1+\widetilde{T}_1$.

\item[vi)] The sesquilinear map 
\begin{equation}\label{eq:D}
\iscp uv \;:=\; \scp{T_1u}{v} 
- \scp{u}{\widetilde{T}_1v} \;,
\quad u,v\in\lW \,, 
\end{equation}
is an indefinite inner product on $\lW$ (cf.~\cite{Bo}) and we have $\lW^{[\perp]}=\lW_0$ and $\lW_0^{[\perp]}=\lW$, where
the $\iscp\cdot\cdot$-orthogonal complement of a set $X\subseteq \lW$
is defined by
\begin{equation*} 
X^{[\perp]} := \bigl\{u\in \lW : (\forall v\in X) \quad 
\iscp uv = 0\bigr\}
\end{equation*}
and it is closed in $\lW$. Moreover, $X^{[\perp][\perp]}=X$ if and only if 
$X$ is closed in $\lW$ and $\lW_0\subseteq X$. 

For future reference, let us define
\begin{equation}\label{eq:lWposneg}
\begin{aligned}
	\lW^+ &:= \bigl\{u\in\lW : \iscp{u}{u}\geq 0\bigr\} \\
	\lW^- &:= \bigl\{u\in\lW : \iscp{u}{u}\leq 0\bigr\} \,.
\end{aligned}
\end{equation}
Note that $X\subseteq X^{[\perp]}$ implies $X\subseteq \lW^+\cap\lW^-$.
\end{enumerate}

Assume, in addition, {\rm (T3)}, i.e.~$(T_0,\widetilde{T}_0)$ is 
a joint pair of abstract Friedrichs operators. Then
\begin{itemize}
\item[vii)] $(\overline{T}_0,\overline{\widetilde{T}}_0)$ satisfies {\rm (T3)} with the same constant $\mu$.
\item[viii)] A lower bound of $\overline{T_0+\widetilde{T}_0}$ is $2\mu>0$.
\item[ix)] We have
\begin{equation}\label{eq:decomposition}
\lW \;=\; \lW_0 \dotplus \ker T_1 \dotplus \ker\widetilde T_1 \;,
\end{equation}
where the sums are direct and all spaces on the right-hand side 
are pairwise $\iscp{\cdot}{\cdot}$-orthogonal. 
Moreover, the linear projections
\begin{equation}\label{eq:projections}
p_\mathrm{k} : \lW \to \ker T_1 \quad \hbox{and} \quad
p_\mathrm{\tilde k}:\lW\to \ker \widetilde{T}_1
\end{equation}
are continuous as maps $(\lW,\|\cdot\|_{T_1})\to (\lH,\|\cdot\|)$, i.e.~$p_\mathrm{k}, p_\mathrm{\tilde k}\in\lL(\lW,\lH)$.
\item[x)] Let $\lV$ be a subspace of the graph space $\lW$ such that
$\lW_0\subseteq\lV\subseteq \lW^+$ (see \eqref{eq:lWposneg}). Then 
$$
(\forall u\in \lV) \qquad \|T_1u\|\geq \mu\|u\| \,.
$$ 
In particular, $\overline{\ran (T_1|_\lV)}=\ran \overline{T_1|_\lV}$.

Analogously, if $\widetilde{\lV}$ is a subspace of $\lW$ such that 
$\lW_0\subseteq\widetilde{\lV}\subseteq\lW^-$, then 
$\|\widetilde T_1 v\|\geq \mu\|v\|$, $v\in\widetilde{\lV}$, 
and $\overline{\ran (\widetilde{T}_1|_{\widetilde{\lV}})}=
\ran \overline{\widetilde{T}_1|_{\widetilde \lV}}$.

\item[xi)] Let $\lV\subseteq\lW$ be a closed subspace (in $\lW$) containing $\lW_0$. 
Then operators $T_1|_\lV$ and $\widetilde{T}_1|_{\widetilde{\lV}}$ are 
mutually adjoint, i.e.~$(T_1|_\lV)^*=\widetilde{T}_1|_{\widetilde{\lV}}$
and $(\widetilde{T}_1|_{\widetilde{\lV}})^*=T_1|_\lV$, if and only if
$\widetilde{\lV}=\lV^{[\perp]}$.
\item[xii)]  Let $\lV\subseteq\lW$ be a closed subspace containing $\lW_0$
such that $\lV\subseteq \lW^+$ and $\lV^{[\perp]}\subseteq\lW^-$.
Then $T_1|_\lV:\lV\to\lH$ and $\widetilde{T}_1|_{{\lV}^{[\perp]}}:\lV^{[\perp]}\to\lH$ are bijective,
i.e.~isomorphisms when we equip their domains with the 
graph topology, and for every $u\in\lV$ the following estimate holds:
\begin{equation}\label{eq:apriori}
	\|u\|_{T_1} \leq \Bigl(1+\frac{1}{\mu}\Bigr) \|T_1 u\| \,.
\end{equation}
The same estimate holds for $\widetilde{T}_1$ and ${\lV}^{[\perp]}$ replacing $T_1$ and $\lV$, respectively.

These bijective realisations of $T_0$ and $\widetilde{T}_0$ we call 
\emph{bijective realisations with signed boundary map}. 
\item[xiii)] Let $\lV\subseteq\lW$ be a closed subspace containing $\lW_0$.
Then $T_1|_\lV:\lV\to\lH$ is bijective if and only if $\lV\dotplus\ker T_1 = \lW$.
\end{itemize}
\end{theorem}

The statements i)--iv), vii) and viii) follow easily from the corresponding
assumptions (cf.~\cite{AEM-2017, EGC}).
The claims v), x) and xii) are already 
argued in the first paper on abstract Friedrichs operators \cite{EGC} for real vector spaces
(see sections 2 and 3 there), while in \cite{ABCE} 
the arguments are repeated in the complex setting.
The same applies for vi) with a remark that for a further structure 
of indefinite inner product space $(\lW, \iscp{\cdot}{\cdot})$ we refer to 
\cite{ABcpde}. 
The decomposition given in ix) is derived in \cite[Theorem 3.1]{ES22},
while for additional claims on projectors we refer to the proof of Lemma 3.5 in the 
aforementioned reference. In the same reference one can find the proof 
of part xiii) (Lemma 3.10 there).
Finally, a characterisation of mutually self-adjointness xi) is obtained 
in \cite[Theorem 9]{AEM-2017}.

\subsection{Additional remarks}

Let us close this section with a few remarks concerning the statements of 
the previous theorem. 


\begin{remark}\label{rem:Vcond}
Not all bijective realisations, characterised in part xiii), are bijective realisation with signed 
boundary map (described in part xii)), 
i.e.~$\lV\subseteq\lW^+$ and $\lV^{[\perp]}\subseteq\lW^-$ is only a sufficient condition.
However, we have the following equivalence: $T_1|_\lV$ is bijective with $\lV\subseteq\lW^+$ if and only if $\widetilde{T}_1|_{\lV^{[\perp]}}$
is bijective with $\lV^{[\perp]}\subseteq\lW^-$. Thus, there is no need in considering pairs of bijective 
realisations with signed boundary map, but we can denote (in this case) each of $T_1|_\lV$ and $\widetilde{T}_1|_{\lV^{[\perp]}}$
as a bijective realisation with signed boundary map of $T_1$ and $\widetilde{T}_1$, respectively. 
Let us comment on this. 

Since $(T_1|_\lV)^*=\widetilde{T}_1|_{\lV^{[\perp]}}$ (see part xi)) and each of spaces $\lV$, $\lV^{[\perp]}$ and $\lH$
is complete, we have equivalence on the level of bijectivity. 
Let us assume that $\lV\subseteq\lW^+$ and let us denote by $\lV_1$ a maximal subspace of $\lW$ such that
$\lV\subseteq\lV_1\subseteq\lW^+$ (it exists by Zorn's lemma; cf.~\cite[Section I.6]{Bo}).
Then $\lV_1^{[\perp]}\subseteq \lW^-$ (cf.~\cite[Lemma 6.3]{Bo} or \cite[Theorem 2(b)]{ABcpde}
for a perspective in the context of abstract Friedrichs operators), hence $T_1|_{\lV_1}$ is bijective as well. 
Since $\lV\subseteq\lV_1$, it must be $\lV=\lV_1$. Hence, $\lV^{[\perp]}=\lV_1^{[\perp]}\subseteq\lW^-$.
The opposite implication can be proved analogously. 

The previous argument actually shows that a subspace $\lV\subseteq\lW^+$ with the property of
$T_1|_\lV$ being bijective is maximal nonnegative subspace, i.e.~if for a subspace $\lV_1\subseteq\lW$ we have $\lV\subseteq\lV_1\subseteq\lW^+$, then it must be $\lV=\lV_1$. Then it is known that such $\lV$
defines a pair of bijective realisations with signed boundary map (see \cite[Theorem 2]{ABcpde}).

Let us close this remark by recalling that for any joint pair of abstract Friedrichs operators there 
exist bijective realisations with signed boundary map 
(see \cite[Theorem 13(i)]{AEM-2017} and \cite[Corollary 3.2]{ES22}).
\end{remark}

\begin{remark}\label{rem:degen}
In part vi) it is commented that for any subset $X$ of $\lW$ we have
that $X\subseteq X^{[\perp]}$ implies $X\subseteq \lW^+\cap\lW^-$.
On the other hand, if $X$ is a subpsace, then the converse holds as well
by the means of the polarisation formula (cf.~\cite[(2.3)]{Bo}).
\end{remark}

\begin{remark}
	If $\lV$ is a subspace of $\lW$ satisfying $\lV=\lV^{[\perp]}$, then $\lV\subseteq \lW^+\cap\lW^-$ (see part vi)).
	Thus, $\lV$ fulfills all assumptions of part xii) implying that 
	the corresponding realisations are bijective realisations with signed boundary map. 
	Since such realisations are of particular interest (see Introduction),
	we will pay attention to examining when such realisations exist
	(see corollaries \ref{cor:eqdom} and \ref{cor:eqdom-number}).
\end{remark}

\begin{remark}\label{rem:kernels_hilbert}
Note that in part ix) of the previous theorem we can also consider the graph norm 
in the codomain of projections, since the graph norm and the (standard) 
norm are equivalent on the kernels $\ker T_1$ and $\ker\widetilde{T}_1$.
Moreover, another equivalent norm on $\ker T_1$ is $\sqrt{-\iscp{\cdot}{\cdot}}$
(for $\nu,\mu\in\ker T_1$ we have $-\iscp{\nu}{\mu}=\scp{\nu}{(\overline{T_0+\widetilde{T}_0})\mu}$),
while on $\ker \widetilde{T}_1$ we can take  $\sqrt{\iscp{\cdot}{\cdot}}$.
In particular, $(\ker T_1, -\iscp{\cdot}{\cdot})$ and $(\ker \widetilde{T}_1,\iscp{\cdot}{\cdot})$ are Hilbert spaces. 

A more detailed point of view can be found in \cite{ABcpde} where the authors recognised that the 
quotient space $\widehat{\lW}:=\lW/\lW_0$ is a Kre\u\i n space (cf.~\cite{Bo}), while the pair of subspaces 
$(\widehat{\ker \widetilde{T}_1}, \widehat{\ker T_1})$ represents a fundamental decomposition 
of $\widehat{\lW}$ (cf.~\cite[remarks 2.13(iii) and 3.3(ii)]{ES22}).
\end{remark}

\begin{remark}\label{rem:lWplus}
It holds $\lW^+\cap\ker T_1=\lW^-\cap\ker\widetilde{T}_1=\{0\}$.
Indeed, let $\nu \in  \lW^+\cap\ker T_1$. Then $-\iscp{\nu}{\nu}\leq 0$,
but since $(\ker T_1, -\iscp{\cdot}{\cdot})$ is a Hilbert space,
we get $-\iscp{\nu}{\nu}=0$, implying $\nu=0$. The second identity is proved in 
the same way. 
\end{remark}

\begin{remark}
A pair of operators satisfying (T1), and thus any pair of abstract Friedrichs operators, is a special case of \emph{dual} or \emph{symmetric} pairs of operators, where in general it is not required that the operators have the same domains (cf.~\cite{Grubb-1968, JP17}). However, the assumption on common domains is not overly restrictive from the perspective of applications to partial differential equations. Indeed, in this setting a standard choice for the ambient Hilbert space $\lH$ is the Lebesgue space of square integrable functions, while for the domain $\lD$ of minimal operators the space of smooth functions with compact support (cf.~\cite{ABcpde, EGC}). 
\end{remark}

\section{Characterisation of abstract Friedrichs operators}\label{sec:charact}

\subsection{Characterisation}

Abstract Friedrichs operators were introduced in \cite{EGC}, while its reformulation completely 
in the spirit of Hilbert spaces was provided in \cite{AEM-2017}. 
Now we make a step forward by deriving the following simple and explicit characterisation that will allow us to connect the theory of abstract Friedichs operators with the well-established theory for (skew-)symmetric operators.

\begin{theorem}\label{thm:abstractFO_L0+S}
A pair of densely defined operators $(T_0,\widetilde{T}_0)$ on $\lH$ 
satisfies {\rm (T1)} and {\rm (T2)} if and only if there exist a densely defined skew-symmetric operator $L_0$ and a bounded self-adjoint operator $S$,
both on $\lH$, such that
\begin{equation}\label{eq:abstractFO_L0+S}
	T_0=L_0+S \qquad \hbox{and} \qquad \widetilde{T}_0=-L_0+S \,.
\end{equation}
For a given pair, the decomposition \eqref{eq:abstractFO_L0+S}
is unique. 

If in the above we include condition {\rm (T3)}, then the same holds with $S$ being strictly positive, i.e.
\begin{equation*}
	\scp{S u}{u} \geq \mu\|u\|^2 \;, \quad u\in\lH\,,
\end{equation*}
where $\mu>0$ is the constant appearing in {\rm (T3)}.
\end{theorem}

\begin{proof}
Let $(T_0,\widetilde{T}_0)$ satisfies (T1) and (T2).
First, we will comment on the uniqueness.
Let $L_0, L_0'$ and $S,S'$ be two densely defined skew-symmetric operator and two bounded self-adjoint operators, respectively, such that 
$T_0 = L_0+S = L_0'+S'$.
Since $\dom T_0\subseteq \dom L_0\cap \dom L_0'$, the operator $L_0-L_0'$ is densely defined and both skew-symmetric and symmetric (since coincides with $S'-S$). Thus, it is necessarily equal to the zero operator, implying $L_0=L_0'$ and $S=S'$ (note that the boundedness of $S$ did not play any role).

Let us now proceed with the existence of such operators $L_0$ and $S$. We define 
$S:= \frac{1}{2}(\overline{T_0+\widetilde{T}_0})$, which is a bounded and self-adjoint operator by Theorem \ref{thm:abstractFO-prop}(v). 
Therefore, by
\begin{equation*}
T_0= \frac{T_0-\widetilde{T}_0}{2} + \frac{T_0+\widetilde{T}_0}{2} 
\qquad \hbox{and}\qquad 
\widetilde T_0= -\frac{T_0-\widetilde{T}_0}{2} + \frac{T_0+\widetilde{T}_0}{2} 
\end{equation*}
it is left to prove that $L_0:=\frac{T_0-\widetilde{T}_0}{2}$ is 
skew-symmetric.
Since $\dom L_0 =\dom T_0\cap\dom\widetilde{T}_0=\lD$,
$L_0$ is densely defined. Furthermore, using $L_0=S-\widetilde{T}_0$
and the boundedness of $S$ we get
$$
L_0^* = S^* - \widetilde{T}_0^* = S - T_1 = S|_\lW -T_1
= \frac{T_1+\widetilde{T}_1}{2} - T_1
= -\frac{T_1-\widetilde{T}_1}{2} \supseteq -L_0 \,,
$$
where in addition we used the second part of Theorem \ref{thm:abstractFO-prop}(v).
Finally, if condition (T3) is satisfied 
as well, $S$ has a positive lower bound by Theorem \ref{thm:abstractFO-prop}(viii).

The converse follows easily by direct inspection. 
\end{proof}

\begin{remark}\label{rem:adjoint}
	From the proof of the previous theorem it is easy to see that for general 
	mutually adjoint closed realisations $(T_1|_{\lV},\widetilde{T}_1|_{\lV^{[\perp]}})$
	(see Theorem \ref{thm:abstractFO-prop}(xi))
	we have
	$$
	\bigl((T_1-\widetilde T_1)|_{\lV}\bigr)^* = -(T_1-\widetilde T_1)|_{\lV^{[\perp]}} \,.
	$$
	Note that for $\lV=\lW_0$ we have the identity
	obtained in the proof of the previous theorem.  
\end{remark}

\begin{remark}
Even for a pair of densely defined operators $(T_0,\widetilde{T}_0)$ on $\lH$ satisfying solely (T1),
we can get decomposition \eqref{eq:abstractFO_L0+S}, but with $S$ being only densely defined and symmetric.
Of course, the decomposition is provided for $L_0=(T_0-\widetilde{T}_0)/2$ and $S=(T_0+\widetilde{T}_0)/2$
and it is unique. Observe that the approach for proving existence of the previous theorem is not appropriate here (since $S$ is not bounded), but instead one just needs to notice that for any $\varphi\in\lD$ we have
$$
\scp{L_0 \varphi}{\varphi} = -\scp{\varphi}{L_0\varphi} \quad \hbox{and} \quad
	\scp{S\varphi}{\varphi}=\scp{\varphi}{S\varphi}\,.
$$
\end{remark}

By Theorem \ref{thm:abstractFO_L0+S}, the study of (pairs of) abstract 
Friedrichs operators is reduced to the study of operators of the form
\eqref{eq:abstractFO_L0+S}, which, in our opinion, makes the situation much more explicit (cf.~\cite[Remark 4.3]{CM21}). 
Let us illustrate some straightforward conclusions. 
For a pair of abstract Friedrichs operators $(T_0,\widetilde{T}_0)$, 
let $L_0$ and $S$ be operators given in
Theorem \ref{thm:abstractFO_L0+S}.
If we denote $L_1:=-L_0^*\supseteq L_0$, then we have
\begin{equation}\label{eq:TLSnotation}
\begin{split}
T_0 &= L_0 + S \;, \\
T_1 &= L_1 + S \;, \qquad
\end{split}
\begin{split}
\widetilde{T}_0 &=-L_0+S \;,\\
\widetilde{T}_1 &=-L_1+S \;.
\end{split}
\end{equation}
In particular, $\lW_0=\dom \overline{L}_0$ and $\lW=\dom L_1$, 
i.e.~spaces $\lW_0$ and $\lW$ are independent of $S$. 
This is also clear by noting that the graph norms 
$\|\cdot\|_{T_1}$ and $\|\cdot\|_{L_1}$ are equivalent, 
due to the boundedness of $S$. 
The same holds for the sesquilinear map \eqref{eq:D}
since
\begin{equation}\label{eq:iscp_L1}
\iscp{u}{v} = \scp{L_1 u}{v}+ \scp{u}{L_1v} \;, \quad u,v\in\lW \,.
\end{equation}
Thus, all conditions on subspaces $\lV\subseteq \lW$ given in 
Theorem \ref{thm:abstractFO-prop}(xii) depend only on $L_1$ (i.e.~$L_0$).
In particular, we can formulate the following corollary.

\begin{corollary}\label{cor:signed_bij}
Let $(T_0,\widetilde{T}_0)$ be a joint pair of abstract Friedrichs operators
on $\lH$ and let $\lV\subseteq\lW$ be a closed subspace containing $\lW_0$
such that $\lV\subseteq\lW^+$ and $\lV^{[\perp]}\subseteq\lW^-$ (with respect to $(T_0,\widetilde{T}_0)$).
For any joint pair of abstract Friedrichs operators $(A_0,\widetilde{A}_0)$ 
on $\lH$ such that 
$$
(A_0-\widetilde{A}_0)^*=(T_0-\widetilde{T}_0)^*
$$ 
we have that $\bigl((\widetilde{A}_0)^*|_{\lV},(A_0)^*|_{\lV^{[\perp]}}\bigr)$
is a pair of bijective realisations with signed boundary map.
\end{corollary}

\begin{remark}\label{rem:m-accretive}
Another perspective to the previous corollary can be made in terms of (linear) \emph{m-accretive} operators \cite[Section 3.3]{Schmudgen}.
To start, note that by the definition of accretive operators on Hilbert spaces for $\lV\subseteq \lW^+$ we have that $L:=L_1|_\lV$ is accretive (see \eqref{eq:lWposneg} and \eqref{eq:iscp_L1}), where we use the notation given in \eqref{eq:TLSnotation}.
If in addition $\lV^{[\perp]}\subseteq\lW^-$ (i.e.~$\lV$ is a maximal nonnegative subspace; see Remark \ref{rem:Vcond}), then $L$ is m-accretive. Indeed, it is sufficient to apply Theorem \ref{thm:abstractFO-prop}(xii) on a pair of abstract Friedrichs operators $(L_0+\mathbbm{1}, -L_0+\mathbbm{1})$ for the identity operator $\mathbbm{1}$ (see Theorem \ref{thm:abstractFO_L0+S}).
Furthermore, since $S$ is bounded and positive, a standard perturbative argument (cf.~\cite[Ch.~3, Corollary 3.3]{Pazy}) implies that $T:=T_1|_\lV$ is m-accretive as well. 
Therefore, the statement of the previous corollary can be seen in the following way: if $L=L_1|_{\lV}$ is m-accretive, then for any $S$ bounded and strictly positive, $L+S$ is also m-accretive. 

On the other hand, if a realisation $T=L+S$ is m-accretive, one can show that then $T$ is a bijective realisation with signed boundary map \cite[Theorem 5.2.2]{Soni24}, and then, by the discussion above, $L$ is also m-accretive.
Thus, we have that the study of bijective realisations with a signed boundary map of abstract Friedrichs operators is tantamount to the study of m-accretive extensions of skew-symmetric operators. This problem of studying m-accretive extensions of skew-symmetric operators has been intensively investigated over the past few decades. For instance, recent studies include
\cite{ACE23, PT24, Trostorff23, WW20}, where various approaches have been applied, as well as different levels of generality (e.g.~linear and/or nonlinear relations). In the following section, we shall make a more precise comparison to our results.
\end{remark}

Not all domains of bijective realisations have the feature described in
the previous corollary, i.e.~there are subspaces $\lV \subseteq \lW$ such that 
realisations $T=L_1|_{\lV}+S$ are bijective for some admissible $S$'es, 
but not all. 
Furthermore, if a subspace $\lV\subseteq\lW$ defines bijective 
realisations for any admissible $S$, that does not imply that 
$\lV\subseteq\lW^+$ and $\lV^{[\perp]}\subseteq\lW^-$.
All this can be illustrated by the following simple example.

\begin{example}\label{exa:1Da}
Let $a<b$, $\lH=L^2((a,b);\R)$ 
(for simplicity we consider only real functions) and $\lD=C^\infty_c(a,b)$.
For $\mu>0$ and $\beta \in L^\infty(a,b)$ such that $\beta\geq\mu$ a.e.~on $(a,b)$,
we consider operators $T_0,\widetilde{T}_0:\lD\to\lH$ given by
$$
T_0u=u'+\beta u \;, \quad \widetilde{T}_0u=-u'+\beta u \,.
$$
Then it is easy to see that $(T_0,\widetilde{T}_0)$ is a joint pair of 
abstract Friedrichs operators, while $\lW= H^1(a,b)$ (which is embedded 
into $C([a,b])$) and $\lW_0=H^1_0(a,b)$.
Of course, in the notation of Theorem \ref{thm:abstractFO_L0+S}, here we have
$L_0u=u'$ and $Su=\beta u$. Then $L_1u := -L_0^*u=u'$ (here the derivative is in the weak sense), $T_1=L_1+S$, $\widetilde{T}_1=-L_1+S$, and
$$
\iscp{u}{v} = u(b)v(b)-u(a)v(a) \;, \quad u,v\in\lW\,.
$$
Let us comment on all bijective realisations of operators $T_0$ and $\widetilde{T}_0$,
i.e.~all bijective restrictions of $T_1$ and $\widetilde{T}_1$.

We define closed subspaces $\lV_\alpha\subseteq \lW$, $\alpha\in \R\cup\{\infty\}$ (here we identify $-\infty$ and $+\infty$), by
\begin{equation*}
	\lV_\alpha := \bigl\{u\in\lW : u(b)=\alpha u(a)\bigr\} \;, \quad \alpha\in\R \,,
\end{equation*}
and $\lV_\infty :=  \bigl\{u\in\lW : u(a)=0 \bigr\}$.
Since $(L_1|_{\lV_\alpha})^*=-L_1|_{\lV_{\frac{1}{\alpha}}}$
(this can be verified by direct calculations), 
we have $(T_1|_{\lV_\alpha})^*=\widetilde{T}_1|_{\lV_{\frac{1}{\alpha}}}$,
where we use the convention: $\frac{1}{\infty} = 0$ and $\frac{1}{0} =\infty$.
Thus, we want to see for which values of $\alpha$,
\begin{equation}\label{eq:ex_Valpha}
\bigl(T_1|_{\lV_\alpha},\widetilde{T}_1|_{\lV_{\frac{1}{\alpha}}}\bigr) 
\end{equation}
is a pair of mutually adjoint bijective realisations (with signed boundary map). 

By \cite[Remark 5.1]{ES22} we have that \emph{all} mutually adjoint bijective realisations are given for $\alpha\in \R\cup\{\infty\}\setminus \{\alpha_\beta\}$,
where $\alpha_\beta:=e^{-\int_a^b\beta(y)dy}$.
Indeed, since 
$\ker T_1=\operatorname{span}\bigl\{e^{-\int_a^x \beta(y)dy}\bigr\}$
we have $\ker T_1\subseteq \lV_{\alpha_\beta}$.
Note that for $\alpha\in\{-1,1\}$ we have $\lV_\alpha=\lV_{\frac{1}{\alpha}}$, hence
$L_1|_{\lV_\alpha}$ is skew-selfadjoint.

By direct inspection we get that only for $\alpha\in (-1,1)$ 
bijective realisations are not 
with signed boundary map (one can also consider $\beta\equiv\mu$ by 
Corollary \ref{cor:signed_bij} and read the result from \cite[Example 1]{ABmn}).
This is in the correspondence with the above result since $\alpha_\beta\in (-1,1)$.
More precisely, we have $\alpha_\beta\in (0,1)$ and by varying $\mu$ and $\beta$
one can get any number in that interval for $\alpha_\beta$. 

Therefore, for $\alpha\in \R\cup\{\infty\}\setminus (-1,1)$ the corresponding 
domains, i.e.~boundary conditions, give rise to bijective realisations 
independent of the choice of admissible $\beta$ (see Corollary \ref{cor:signed_bij}).
The same holds for $\alpha\in (-1,0]$ although these bijective realisations
are not with signed boundary map. Here the reason lies in the fact that 
$\lV_\alpha\cap\ker T_1=\{0\}$ (cf.~Theorem \ref{thm:abstractFO-prop}(xiii))
for any $\alpha\in (-1,0]$ and any choice of
admissible $\beta$.
Finally, there is no $\alpha$ in $(0,1)$ with this property. However, for 
fixed $\beta$, all $\alpha\in (0,1)$ but one ($\alpha=\alpha_\beta$) 
correspond to mutually adjoint bijective realisations \eqref{eq:ex_Valpha}.

We will return to this example to consider general symmetric parts.
\end{example}

\begin{remark}\label{rem:L0+C}
If we consider $T_0=L_0+C$ where $L_0$ is skew-symmetric and $C$ bounded (i.e.~$C$ is not necessarily self-adjoint), then 
it is easy to see that the discussion preceding Corollary \ref{cor:signed_bij} still holds. 
More precisely, spaces $\lW_0$, $\lW$ and indefinite inner product $\iscp{\cdot}{\cdot}$
are independent of $C$ (\eqref{eq:iscp_L1} holds precisely as it is; recall that $L_1=-L_0^*$) and
the graph norm is equivalent with $\|\cdot\|_{L_1}$ (cf.~\cite[Subsection 2.2]{BEW23}).
Thus, these objects depend solely on the unbounded component of the skew-symmetric part of $T_0$. 
\end{remark}

\subsection{Deficiency indices}

The previous example illustrates that in order to get all bijective realisations
(not only with signed boundary map) it is not enough to consider $L_0$ alone, 
we must also bring the symmetric part $S$ into play.
In particular, information on kernels $\ker T_1$ and $\ker \widetilde{T}_1$ is 
essential. 
By Theorem \ref{thm:abstractFO-prop}(ix) we have 
\begin{equation*}
	\lW \;=\; \lW_0 \dotplus \ker T_1 \dotplus \ker\widetilde T_1 \;.
\end{equation*}
Hence, the sum of dimensions of the kernels is constant and equals the 
codimension of $\lW_0$ in $\lW$. However, from here we cannot conclude that
(cardinal) numbers $\dim \ker T_1$ and $\dim\ker\widetilde{T}_1$ are independent of
$S$, where $T_1=L_1+S$. 
If so, this would be beneficial in the analysis (see Example \ref{exa:1Db} below). 
Let us motivate why one should expect such a result. 
Since $L_0$ is skew-symmetric, we have that $-iL_0$ is symmetric. 
Thus, for any positive constant $\beta>0$ we have
$$
\dim\ker (L_1+\beta\mathbbm{1}) = \dim\ker (iL_0^*-i\beta\mathbbm{1})
	= \dim\ker ((-iL_0)^*-i\beta\mathbbm{1})= d_+(-iL_0)  \,,
$$
where $\mathbbm{1}$ denotes the identity operator and on the right we have the deficiency index (or the defect number)
of $-iL_0$, which is known to be independent of $\beta>0$ 
(see \cite[Section 3.1]{Schmudgen}).
Analogously, $\dim\ker (L_1-\beta\mathbbm{1})= d_-(-iL_0)$.
Therefore, all that we need is to show that instead of $\beta\mathbbm{1}$ 
we can put an arbitrary bounded self-adjoint strictly positive
operator. 
Below is a slightly more general statement.

\begin{lemma}\label{lem:defect}
Let $L_0$ be a densely defined skew-symmetric operator and let 
us denote $L_1:=-L_0^*$. For a bounded linear operator $C$ with strictly positive 
symmetric part $\frac{1}{2}(C+C^*)$, we define
\begin{equation*}
d^C_+(L_0):=\dim\ker (L_1+C) \quad \hbox{and} \quad	
	d^C_-(L_0):=\dim\ker (L_1-C) \,. 
\end{equation*} 
Then $d^C_+(L_0)$ and $d^C_-(L_0)$ are independent of $C$,
i.e.~$d_\pm^C(L_0)=d_\pm(-iL_0)$. 
\end{lemma} 

\begin{proof}
Since $L_0$ is closable and $d_\pm^C(L_0)=\dim\ker (L_1\pm C)=d_\pm^C(\overline{L_0})$, 
we can assume that $L_0$ is closed.
We shall prove the claim for $d_+^C(L_0)$, while the same argument 
applies on $d_-^C(L_0)$. 

Let us take arbitrary bounded operators
$C$ and $C'$ with strictly positive symmetric parts, 
and let us denote by $\mu$ and $\mu'$ the greatest
lower bounds of their symmetric parts, respectively. 
We shall first argue in a specific situation when $\|C-C'\|<\min\{\mu,\mu'\}$,
where here $\|\cdot\|$ denotes the operator norm. 
Before we start, let us note that according to Theorem \ref{thm:abstractFO_L0+S},
both operators $-L_0+C^*$ and $-L_0+(C')^*$ define a (pair of) abstract Friedrichs operators (the skew-symmetric part is equal to the sum of $L_0$ (unbounded part) and the skew-symmetric component of $C$ or $C'$ (bounded part) with the negative sign). 
Hence, all results of Theorem \ref{thm:abstractFO-prop} are applicable. 
In particular, since $L_0$ is assumed to be closed, 
part (x) of the aforementioned theorem implies that 
the ranges of the operators $-L_0 + C^*$ and $-L_0 + (C')^*$ are closed.

If $d_+^C(L_0) > d_+^{C'}(L_0)$, then there exists 
$0\neq \nu\in \ker (L_1+C)\cap\ker(L_1+C')^\perp$
(cf.~\cite[Lemma 2.3]{Schmudgen}).
Since $\ran(-L_0+(C')^*)$ is closed, we have $\ker(L_1+C')^\perp=\ran (-L_0+(C')^*)$.
Thus, we also have $\nu\in \ran(-L_0+C^*)^\perp\cap\ran(-L_0+(C')^*)$. 
Let $0\neq u\in\dom L_0$ be such that $\nu=(-L_0+(C')^*)u$.
Then it holds
\begin{equation}\label{eq:d+d-}
\scp{(-L_0+(C')^*)u}{(-L_0+C^*)u} = 0 \,.
\end{equation}

Applying Theorem \ref{thm:abstractFO-prop}(x) and the identity above, we have
\begin{align*}
\mu'\|u\| \|(-L_0+(C')^*)u\| &\leq \|(-L_0+(C')^*)u\|^2 \\
&= \scp{(-L_0+C^*)u+(C'-C)^*u}{(-L_0+(C')^*)u} \\
&= \scp{u}{(C'-C)(-L_0+(C')^*)u}\\
&\leq \|u\|\|C'-C\|\|(-L_0+(C')^*)u\| \,.
\end{align*}
Since $u\neq 0$ and $\nu=(-L_0+(C')^*)u\neq 0$, this implies $\|C'-C\|\geq \mu'$, 
which contradicts the starting assumption $\|C'-C\|<\mu'$. 
Hence, it should be $d_+^C(L_0) \leq d_+^{C'}(L_0)$.

Since the identity \eqref{eq:d+d-} is symmetric with respect to $C$ and $C'$, 
the same holds even if we start with the assumption $d_+^C(L_0) < d_+^{C'}(L_0)$.
However, then we have $(-L_0+C^*)u\neq 0$. Repeating the last calculations 
with $C'$ and $C$ swapping places, we come to the analogous conclusion:
$\mu\leq \|C-C'\|<\mu$. Therefore, it must be $d_+^C(L_0) = d_+^{C'}(L_0)$.

It is left to prove the statement without the additional assumption 
$\|C-C'\|<\min\{\mu,\mu'\}$. This easily follows by noting that the set of all bounded operators on $\lH$ with strictly positive symmetric part is convex. More precisely, 
for each $\lambda\in [0,1]$ we have that 
$C_\lambda:=\lambda C+(1-\lambda)C'$ is bounded and 
the greatest lower bound of its symmetric part is $\lambda\mu+(1-\lambda)\mu'\geq \min\{\mu,\mu'\}$.
Moreover, $\|C_{\lambda_1}-C_{\lambda_2}\| = |\lambda_1-\lambda_2|\|C-C'\|$.
Thus, we can pick finitely many values $0=\lambda_1<\lambda <\dots<\lambda_m=1$ such that
$\|C_{\lambda_j}-C_{\lambda_{j+1}}\|<\min\{\mu,\mu'\}$, $j=1,2,\dots,,m-1$. 
Therefore, by applying the previously obtained result, we get
$$
d_+^{C'}(L_0)=d_+^{C_{\lambda_1}}(L_0)=d_+^{C_{\lambda_2}}(L_0)=\dots=d_+^{C_{\lambda_m}}(L_0)
	=d_+^{C}(L_0) \,,
$$
concluding the proof. 
\end{proof}

\begin{remark}
The stability of the indices $d_\pm^C(L_0)$ with respect to $C$ is, in fact, a consequence of the homotopy argument. Indeed, for a fixed $C$, let us consider $H_\pm(C',\lambda)= d_\pm^{(1-\lambda)C+\lambda C'}(L_0)$, where $C$ runs through all bounded linear operators with a strictly positive symmetric part and $\lambda\in[0,1]$. Then one just needs to show that $H_\pm$ behave well (i.e.~in a continuous manner) with respect to $\lambda$, which is precisely what we studied in the proof.
\end{remark}

\begin{remark}
For a densely defined skew-symmetric operator $L_0$ we will refer to the cardinal numbers $d_\pm(-iL_0)$ as the deficiency indices (or the defect numbers) of $L_0$, 
and we introduce the notation $d_\pm(L_0):=d_\pm(-iL_0)$.
The definition is not ambiguous because depending on whether the operator is 
symmetric or skew-symmetric the corresponding definition applies.
\end{remark}

Let us return to the analysis of Example \ref{exa:1Da}.
 
\begin{example}\label{exa:1Db}
In Example \ref{exa:1Da} we studied specific (multiplicative) symmetric parts.
Let us now consider a general situation where 
$$
T_0u=u'+C u \;, \quad \widetilde{T}_0u=-u'+C u \,,
$$
for an arbitrary bounded linear operator $C$ with strictly positive 
symmetric part $\frac{1}{2}(C+C^*)$.

First note that by Example \ref{exa:1Da} (cf.~\cite[Subsection 6.1]{AEM-2017}) 
and Lemma \ref{lem:defect} we have
$\dim \ker T_1 = \dim \widetilde{T}_1 = 1$ (for any admissible $C$).
 
The conclusion of Example \ref{exa:1Da} for the range $\R\cup\{\infty\}\setminus (-1,1)$ remains the same (cf.~Corollary \ref{cor:signed_bij}), i.e.~for 
these values of $\alpha$ we get for any admissible $C$ bijective realisations
(even with signed boundary map). 

Let us take $\alpha\in (-1,0)$. Since the codimension of $\lV_\alpha$ in $\lW$ 
equals 1, and $\dim\ker T_1=1$, 
by Theorem \ref{thm:abstractFO-prop}(xiii) it is sufficient to prove that
$\lV_\alpha$ and $\ker T_1=\operatorname{span}\{\varphi_C\}$ have a trivial intersection 
to get that the corresponding realisations are bijective. 
Let us assume on the contrary that $\varphi_C\in \lV_\alpha$. 
Since $\alpha<0$, we have $\varphi_C(a)\varphi_C(b)<0$. Thus, recalling 
that $\lW\hookrightarrow C([a,b])$, there exists $c\in (a,b)$ such that 
$\varphi_C(c)=0$. Moreover, $\varphi_C\in\ker T_1$ implies that 
$\varphi_C'+C\varphi_C =0$ in $(c,b)$ as well. This together with $\varphi_C(c)=0$
implies that $\varphi_C\equiv 0$ in $(c,d)$. Indeed, just recall that $\lV_\infty$
defines a bijective realisation. In particular, we have $\varphi_C(b)=0$, implying 
$\alpha=0$, which is a contradiction. 
Therefore, for any $\alpha\in (-1,0)$ we get bijective realisations independently 
of the choice of $C$. 
Note that here we were not able to capture the value $\alpha=0$. 

The argument given in Example \ref{exa:1Da} is sufficient to conclude that 
in general there is no $\alpha$ in $(0,1)$ with the property that the pair 
of domains $(\lV_\alpha,\lV_{\frac{1}{\alpha}})$ gives rise to (mutually adjoint) bijective realisations \eqref{eq:ex_Valpha} for any choice of admissible $C$. 
On the other hand, for fixed $C$, since $\dim\ker T_1=1$, there exists precisely one $\alpha=\alpha_C\in [0,1)$ for which the corresponding realisations are not bijective. 
Even though we have bijective realisations for all other values of $\alpha$, i.e.~$\alpha\in [0,1)\setminus\{\alpha_C\}$, this case still depends on $C$.
\end{example}

A particularly interesting case of bijective realisations with signed boundary map
is when $\lV=\lV^{[\perp]}$ (see Introduction). By Remark \ref{rem:adjoint} this occurs if and only 
if the associated realisation of the skew-symmetric part $L_0$ is skew-selfadjoint.
Thus, such subspace $\lV$ exists if and only if $d_+(L_0)=d_-(L_0)$ (see 
\cite[Theorem 13.10]{Schmudgen}). 
Applying Lemma \ref{lem:defect} we can formulate the following corollary.

\begin{corollary}\label{cor:eqdom} 
	Let $(T_0,\widetilde T_0)$ be a joint pair of abstract Friedrichs operators on $\lH$. 
	There exists a closed subspace $\lV$ of $\lW$ with $\lW_0\subseteq \lV$ and such that $(T_1|_{\lV},\widetilde{T}_1|_{\lV})$ is a pair of mutually adjoint bijective realisations related to $(T_0,\widetilde{T}_0)$ if and only if 
	$\ker T_1$ and $\ker\widetilde{T}_1$ are isomorphic.
\end{corollary}

\begin{remark}\label{rem:isom-char}
	The notion of isomorphism of Hilbert spaces used in the previous corollary is the standard (and natural) one
	(cf.~\cite[I.5.1.~Definition]{Conway}): two Hilbert spaces are \emph{isomorphic} if there exists a linear 
	surjective isometry (\emph{isomorphism} or \emph{unitary transformation}) between them.
	
	One can find several characterisations, e.g.~two Hilbert spaces are isomorphic if and only if
	\begin{itemize}
		\item[i)] they have the same dimension.
		\item[ii)] there exists a linear bounded bijection between them. 
	\end{itemize}
	For the first claim we refer to \cite[I.5.4.~Theorem]{Conway}, while in the latter one needs to discuss only 
	the converse. This can be done in a straigtforward constructive way. Indeed, if we denote by $A$ 
	a linear bounded bijection between two given Hilbert spaces, then $U:=A(A^*A)^{-\frac 12}$ is 
	an isomorphism (in the above sense). 
\end{remark}

\begin{remark}\label{rem:iso-ker}
	Since $\ker \widetilde T_1$ is a Hilbert space when both equipped with $\scp\cdot\cdot$ (the standard inner product of the 
	ambient space $\lH$) and $\iscp\cdot\cdot$ (the indefinite inner product on $\lW$), and the identity map 
	$i:(\ker\widetilde T_1,\scp{\cdot}{\cdot})\to (\ker \widetilde T_1, \iscp{\cdot}{\cdot})$ is continuous
	(due to the boundedness of $T_1+\widetilde T_1$ on $\lH$), 
	it is irrelevant which Hilbert space structure we consider on $\ker \widetilde T_1$ in Corollary \ref{cor:eqdom}.
	The same applies on $\ker T_1$ as well, with the only difference that $\iscp\cdot\cdot$
	should be replaced by $-\iscp\cdot\cdot$.
\end{remark}

\section{Classification of the von Neumann type}\label{sec:classif}

\subsection{Preliminaries}

Applying von Neumann's extension theory of symmetric operators (cf.~\cite[Theorem 13.9]{Schmudgen}) on $-iL_0$, we can classify all skew-selfadjoint
(even closed skew-symmetric) realisations of $L_0$ in terms of
unitary transformations between (closed subspaces of) $\ker (L_1+\mathbbm{1})$
and $\ker (L_1-\mathbbm{1})$. 
Of course, by Lemma \ref{lem:defect} (see also Remark \ref{rem:isom-char})
this can also be done when $\ker (L_1+\mathbbm{1})$ and $\ker (L_1-\mathbbm{1})$
are replaced by $\ker T_1$ and $\ker\widetilde{T}_1$.
We are about to focus on this situation since, as it was demonstrated in examples \ref{exa:1Da} and \ref{exa:1Db}, 
often it is desirable to keep abstract Friedrichs operators $T_0$ and $\widetilde{T}_0$ as a whole (e.g.~not all bijective realisations of $T_0=L_0+S$ correspond to skew-symmetric realisations of $L_0$). Also, in terms of partial differential operators (especially with variable coefficients), it is sometimes easier to work with the operator $L_0+C$, for some $C$, than with skew-symmetric operator $L_0$ itself.
However, we will not make use of \cite[Theorem 13.9]{Schmudgen}, but develop an 
independent constructive proof. This will allow for an explicit classification and 
at the same time provide an alternative proof in the symmetric setting (for even more general situations).

Let us start with the following lemma.

\begin{lemma}\label{lm:isometry}
	Let $(T_0,\widetilde{T}_0)$ be a joint pair of abstract Friedrichs operators 
	on $\lH$. Let $\lV\subseteq\lW$ be a closed subspace containing $\lW_0$
	and let us define $\lG:=p_\mathrm{k}(\lV)$ and 
	$\widetilde{\lG}:=p_{\tilde{\mathrm{k}}}(\lV)$, where $p_\mathrm{k}$ and $p_{\tilde{\mathrm{k}}}$ are given by 
	\eqref{eq:projections}.
	Then, we have the following.
	\begin{itemize}
		 
		\item[i)] $T_1|_\lV$ is a bijective realisation of $T_0$ if and only if $\lV\cap \ker T_1=\{0\}$ and $\widetilde{\lG}=\ker \widetilde{T}_1$.
		
		\item[ii)] Let $\lV\cap \ker T_1=\{0\}$. Then $U: \widetilde{\lG}\to \lG$
		defined by
		\begin{equation}\label{eq:iso-ker-U}
		U(p_{\tilde{\mathrm{k}}}(u))=p_\mathrm{k}(u), \quad  u\in \lV\,, 
		\end{equation}
		is a well-defined closed linear map.
		Moreover, $U$ is bounded if and only if $\widetilde{\lG}$ is closed in $\ker \widetilde{T}_1$ (cf.~Remark \ref{rem:iso-ker}). 
		
		
		\item[iii)] If $\lV\subseteq\lW^+$, then $\widetilde{\lG}$ is closed and $U:(\widetilde{\lG},\iscp{\cdot}{\cdot}) \to
		(\ker T_1,-\iscp{\cdot}{\cdot})$ is non-expansive, i.e.~the norm of $U$ with respect to the indicated norms is less than or equal to $1$.
		
		\item[iv)] If $\lV\subseteq \lW^+\cap\lW^-$, then both $\widetilde{\lG}$ and $\lG$ are closed and 
		$U:(\widetilde{\lG},\iscp{\cdot}{\cdot})\to (\lG,-\iscp{\cdot}{\cdot})$
		 is a unitary transformation (cf.~Remark 
		 \ref{rem:isom-char}).

		 \item[v)]  Let $\lV\cap \ker T_1=\{0\}$. Then $\lV$ coincides with $\lV_U$ given 
		 by
		 \begin{equation}\label{eq:V_U}
		 \lV_U:=\bigl\{u_0+U\tilde\nu + \tilde\nu : u_0\in\lW_0, \, \tilde\nu\in\widetilde{\lG}\bigr\} \,,
		 \end{equation}
		 where $U$ is defined by \eqref{eq:iso-ker-U}, 
		 and $T_1|_\lV(u_0+U\tilde\nu + \tilde\nu) = \overline{T}_0 u_0 + (\overline{T_0+\widetilde{T}_0})\tilde\nu$.
		 
		 Moreover, such $U$ is unique, i.e.~if for a subspace $\widetilde{\lG}\subseteq\lW$ and a 
		 closed operator $U:\widetilde{\lG}\to\ker T_1$ we have $\lV=\lV_U$, 
		 where $\lV_U$ is defined by the formula above, then $U$ is given by \eqref{eq:iso-ker-U}.
	\end{itemize}
\end{lemma}

\begin{remark}
Notice that the assumptions in statements ii)-iv) are gradually strengthened 
(see Remark \ref{rem:lWplus}).
Furthermore, by Remark \ref{rem:degen} the assumption of part iv) is equivalent to 
$\lV\subseteq\lV^{[\perp]}$. 
\end{remark}

\begin{remark}
Of course, in part iv) it is implicitly required that $\widetilde{\lG}$ and $\lG$ are isomorphic. 

A trivial situation when the assumption $\lV\subseteq \lW^+\cap\lW^-$ is satisfied occurs for 
$\lV=\lW_0$. Then $\lG=\widetilde{\lG}=\{0\}$, hence they are obviously isomorphic. 
\end{remark}

\begin{remark}\label{rem:UinH}
	By Remark \ref{rem:isom-char}, in the regime $\lV\subseteq\lW^+\cap\lW^-$ of part iv) of the previous lemma
	from the mapping $U$, given by \eqref{eq:iso-ker-U},
	we can construct an isomorphism between spaces $\widetilde\lG$ and $\lG$ when both are equipped with the standard 
	inner product $\scp\cdot\cdot$.
	Indeed, if we denote by $\tilde{\iota}:\widetilde{\lG}\hookrightarrow\lH$ and 
	${\iota}:\lG\hookrightarrow\lH$ canonical embeddings (note that then $\tilde\iota^*$ and $\iota^*$ 
	are orthogonal projections in $\lH$ onto $\widetilde\lG$ and $\lG$, respectively),
	then a unitary transformation is given by
	\begin{equation*}
	U\Bigl(\tilde\iota^* \,\overline{T_0+\widetilde T_0}\,U^*\iota^*(\overline{T_0+\widetilde T_0})^{-1} U\Bigr)^{-\frac 12}\,.
	\end{equation*}
	A trivial situation is when $\overline{T_0+\widetilde T_0}=\alpha\mathbbm{1}$, for some $\alpha\in\C$, since then 
	the expression above equals $U$.
\end{remark}

\begin{remark}\label{rem:V_U_closed}
It is clear that for any given $U:\widetilde{\lG}\to\ker T_1$, $\lV_U$ defined by \eqref{eq:V_U}
is a subspace of $\lW$ which contains $\lW_0$.
Moreover, if $U$ is closed, then $\lV_U$ is closed as well. Indeed, from 
$u_0^n+U\tilde\nu_n+\tilde\nu_n\to u_0+\nu+\tilde\nu\in\lW$ it follows (see the proof of \cite[Lemma 3.5]{ES22})
that $U\tilde\nu_n\to \nu$ and $\tilde\nu_n\to\tilde\nu$. Hence, for closed $U$ we have $\tilde\nu\in\dom U$
and $\nu=U\tilde\nu$, implying that $u_0+\nu+\tilde\nu=u_0+U\tilde\nu+\tilde\nu\in\lV_U$.

\end{remark}

\begin{proof}[Proof of Lemma \ref{lm:isometry}]
	\begin{itemize}
		\item[i)] Let us assume that $T_1|_\lV$ is a bijective realisation.
		Injectivity implies that $\lV\cap\ker T_1=\{0\}$, while 
		$p_{\tilde{\mathrm{k}}}(\lV)\subseteq \ker \widetilde{T}_1$
		is trivial. Let us prove the opposite inclusion. 
		By Theorem \ref{thm:abstractFO-prop}(xiii), we have 
		$\lW=\lV\dot{+}\ker{T}_1$.
		Thus, for any $\tilde \nu\in \ker \widetilde T_1$ (recall that $\ker \widetilde T_1\subseteq \lW$) there exist unique 
		$u\in \lV$ and $\nu\in \ker{T}_1$, such that $\tilde \nu=u+\nu$.
		This implies $u=-\nu+\tilde\nu$, so $\tilde\nu=p_{\tilde{\mathrm{k}}}(u)\in
		p_{\tilde{\mathrm{k}}}(\lV)$.
		
		For the converse, we shall make use of Theorem \ref{thm:abstractFO-prop}(xiii) 
		again. Let us take an arbitrary $w\in\lW$. By Theorem \ref{thm:abstractFO-prop}(ix) there exist unique $w_0\in\lW_0$, $\nu\in\ker T_1$
		and $\tilde\nu\in\ker\widetilde{T}_1$ such that
		$w=w_0+\nu+\tilde\nu$.
		The assumption ensures existence of $u\in\lV$ such that
		$u=u_0+\mu+\tilde\nu$, for some $u_0\in\lW_0$ and $\mu\in\ker T_1$.
		By subtracting $\tilde\nu$ from the second equation and inserting it into 
		the first, we get 
		$$
		w = (w_0-u_0+u) + (\nu-\mu) \,.
		$$
		Since the first term on the right hand side belongs to $\lV$ (note that $\lW_0\subseteq \lV$) and the second 
		one to $\ker T_1$, Theorem \ref{thm:abstractFO-prop}(xiii) is applicable

		\item[ii)] We start by showing that $U$ is a well-defined function. Let $u,v\in \lV$ be such that $p_{\tilde{\mathrm{k}}}(u)=p_{\tilde{\mathrm{k}}}(v)$.  By the decomposition given in Theorem \ref{thm:abstractFO-prop}(ix), there exist $u_0,v_0\in \lW_0$, such that 
		\begin{align*}
		u = u_0+p_{\mathrm{k}}(u)+p_{\tilde{\mathrm{k}}}(u)\,,\quad 
		v = v_0+p_{\mathrm{k}}(v)+p_{\tilde{\mathrm{k}}}(v)\;.
		\end{align*}
		Thus,
		\begin{align*}
		u-v=(u_0-v_0)+\bigl(p_{\mathrm{k}}(u)-p_{\mathrm{k}}(v)\bigr)\;.
		\end{align*}
		Since $(u-v)-(u_0-v_0)\in \lV$, we get
		$p_{\mathrm{k}}(u)-p_{\mathrm{k}}(v) \in \lV \cap \ker T_1 =\{0\}$.
		Hence, $U$ is well-defined. 
		
		Linearity follows from the linearity of projections $p_{\mathrm{k}}$ and $p_{\tilde{\mathrm{k}}}$.
		Let us show that $U$ is closed. Take $(u_n)$ in $\lV$ such that 
		$p_{\tilde{\mathrm{k}}}(u_n)\to \tilde\nu$ in $\ker \widetilde{T}_1$ and 
		$Up_{\tilde{\mathrm{k}}}(u_n)=p_{\mathrm{k}}(u_n)\to \nu$ in $\ker T_1$ 
		(as $n$ tends to infinity). This implies that $\bigl(p_{\tilde{\mathrm{k}}}(u_n)+p_{\mathrm{k}}(u_n)\bigr)$
		converges to $\tilde\nu+\nu$ in $\lW$. Since $\lW_0\subseteq \lV$, for each $n\in\N$ 
		we have $p_{\tilde{\mathrm{k}}}(u_n)+p_{\mathrm{k}}(u_n)\in\lV$. 
		Thus, by the closedness of $\lV$, we obtain $u:=\tilde\nu+\nu\in\lV$, implying
		$\tilde\nu=p_{\tilde{\mathrm{k}}}(u)\in p_{\tilde{\mathrm{k}}}(\lV)$ and 
		$U \tilde\nu = U p_{\tilde{\mathrm{k}}}(u) =p_{\mathrm{k}}(u)=\nu$.
		
		Similarly as with the closedness of $U$, in the last part the goal is to exploit the fact that 
		$\lV$ is closed.
		Indeed, let us assume that $U$ is bounded, i.e.~there exists $c>0$ such that $\|p_{\mathrm{k}}(u)\|\leq c\|p_{\tilde{\mathrm{k}}}(u)\|$, $u\in\lV$.
		Let us take $(u_n)$ in $\lV$ such that 
		$p_{\tilde{\mathrm{k}}}(u_n)\to \tilde\nu$ in $\ker \widetilde{T}_1$, as $n$ tends to infinity.
		Using the boundedness of $U$ we get that the sequence $(p_{\mathrm{k}}(u_n))$ is a Cauchy sequence 
		in $\ker T_1$, hence convergent. Now we get $\tilde\nu\in p_{\tilde{\mathrm{k}}}(\lV)$
		following the previous reasoning. On the other hand, if $p_{\tilde{\mathrm{k}}}(\lV)$ is closed, 
		then $U$ is a closed linear map between two Hilbert spaces $p_{\tilde{\mathrm{k}}}(\lV)$ and 
		$\ker T_1$. Thus, $U$ is bounded by the closed graph theorem. 
		
%

		\item[iii)] Let $u\in \lV\subseteq \lW$. 
		Using $u = u_0+p_{\mathrm{k}}(u)+p_{\tilde{\mathrm{k}}}(u)$,
		$\lW^{[\perp]}=\lW_0$ (see Theorem \ref{thm:abstractFO-prop}(vi)) and $\lV\subseteq\lW^+$, we get
		\begin{equation}\label{eq:Ucont}
		0\leq \iscp{u}{u}=\iscp{p_{\mathrm{k}}(u)}{p_{\mathrm{k}}(u)}+\iscp{p_{\tilde{\mathrm{k}}}(u)}{p_{\tilde{\mathrm{k}}}(u)}\;,
		\end{equation}
		which in terms of the operator $U$ reads
		\begin{equation*}
		-\iscp{Up_{\tilde{\mathrm{k}}}(u)}{Up_{\tilde{\mathrm{k}}}(u)}
		\leq\iscp{p_{\tilde{\mathrm{k}}}(u)}{p_{\tilde{\mathrm{k}}}(u)}\;.
		\end{equation*}
		Thus, $U:\widetilde{\lG} \to\ker T_1$ is bounded (hence $\widetilde{\lG}$ is closed by part ii)) 
		and $\|U\|\leq 1$ (with respect to the norms
		$\sqrt{\iscp{\cdot}{\cdot}}$ and $\sqrt{-\iscp{\cdot}{\cdot}}$, respectively). 
		
		\item[iv)] When $\lV\subseteq\lW^+\cap\lW^-$, then in 
		\eqref{eq:Ucont} we have equality. 
		This allows us to follow the last part of the proof of part ii) to conclude that 
		$\lG$ is closed as well. 
		Furthermore, $U$ is obviously a unitary transformation between Hilbert spaces
		$(\widetilde{\lG},\iscp{\cdot}{\cdot})$ and $(\lG,-\iscp{\cdot}{\cdot})$.

		\item[v)] Since for any $u\in\lV$ there exists $u_0\in\lW_0$ such that 
		$u=u_0+p_{\mathrm{k}}(u)+p_{\tilde{\mathrm{k}}}(u)=u_0+Up_{\tilde{\mathrm{k}}}(u)
		+p_{\tilde{\mathrm{k}}}(u)$, it is clear that $\lV=\lV_U$. 
		
		
		For arbitrary $u_0\in\lW_0$ and $\tilde\nu\in\widetilde{\lG}$ we have
		$$
		T_1|_\lV(u_0+U\tilde\nu+\tilde\nu) = T_1 (u_0+U\tilde\nu+\tilde\nu) =
			T_1 u_0 + T_1 \tilde\nu = \overline{T}_0 u_0 + (T_1+\widetilde{T}_1) \tilde\nu \,,
		$$
		where we have used $\lV\subseteq\lW$, $T_1(U\tilde\nu)=0$, $T_1|_{\lW_0}=\overline{T}_0$ 
		and $\widetilde{T}_1\tilde\nu=0$.
		
		Let us take two subspaces $\widetilde{\lG}_i\subseteq\ker \widetilde{T}_1$, $i\in\{1,2\}$, and two 
		closed operators $U_i:\widetilde{\lG}_i\to\ker T_1$, $i\in\{1,2\}$, such that 
		$\lV_{U_1}=\lV_{U_2}$. This means that for an arbitrary $\tilde\nu_1\in\widetilde{\lG}_1$
		there exist $u_0\in\lW_0$ and $\tilde\nu_2\in\widetilde{\lG}_2$ such that
		$$
		U_1\tilde\nu_1+\tilde\nu_1 = u_0 +U_2\tilde\nu_2+\tilde\nu_2 \;.
		$$	
		Applying Theorem \ref{thm:abstractFO-prop}(ix) we get $u_0=0$, $\tilde\nu_1 =\tilde\nu_2$ 
		and $U_1\tilde\nu_1=U_2\tilde\nu_2$. Hence, $U_1\subseteq U_2$. By the symmetry we can conclude
		that in fact we have $U_1=U_2$.  
		This proves that such $U$ is unique, and by the first part we have that $U$ is necessarily
		given by \eqref{eq:iso-ker-U}. 
	\end{itemize}
\end{proof}

\begin{remark}\label{rem:V_Uperp}
	For $\lV_U$ given by \eqref{eq:V_U} we can explicitly write $\lV_U^{[\perp]}$ (i.e.~the domain of 
	$(T_1|_{\lV_U})^*$; see Theorem \ref{thm:abstractFO-prop}(xi)) in terms of the adjoint
	operator $U^*$. For simplicity, let us elaborate on this only in the case of bounded 
	$U:\widetilde{\lG}\to\ker T_1$, i.e.~$\widetilde\lG=p_{\tilde{\mathrm{k}}}(\lV)$ is closed
	(see Lemma \ref{lm:isometry}(ii)). 
	Then $U^*:\ker T_1\to\widetilde{\lG}$ is also bounded and 
	$p_{\mathrm{k}}(\lV^{[\perp]})=\ker T_1$. Moreover, 
	$$
	U^*p_\mathrm{k}(v) = U^*P_{\bar{\lG}} p_\mathrm{k}(v) = \widetilde{P}_{\widetilde{\lG}}p_{\tilde{\mathrm{k}}}(v) \,, \quad v\in\lV^{[\perp]}\,,
	$$
	where $P_{\bar\lG}$ and $\widetilde{P}_{\widetilde{\lG}}$ denote the orthogonal projections on 
	$\bar\lG$ and $\widetilde{\lG}$ within spaces $(\ker T_1,-\iscp{\cdot}{\cdot})$ and $(\ker\widetilde{T}_1,\iscp{\cdot}{\cdot})$, respectively 
	($\bar\lG$ is the closure of $\lG=p_\mathrm{k}(\lV)$ in $\ker T_1$).
	Furthermore, $\lV_U^{[\perp]}$ is then given by
	\begin{align*}
	\lV_U^{[\perp]}:=\bigl\{v_0+ \mu_1 + U^*\mu_1 + \mu_2+\tilde\mu_2 :  v_0 & \in\lW_0, \, 
	\mu_1\in\bar\lG, \\
	&\mu_2\in\lG^{[\perp]}\cap\ker T_1, \, \tilde\mu_2\in\widetilde\lG^{[\perp]}\cap\ker \widetilde{T}_1\bigr\} \,.
	\end{align*}
	Note that $\lG^{[\perp]}\cap\ker T_1$ is just the orthogonal complement of $\lG$ within the space
	$(\ker T_1,-\iscp{\cdot}{\cdot})$, and analogously for $\widetilde\lG^{[\perp]}\cap\ker \widetilde{T}_1$.
	
	As an example, let us consider $\lV=\lW_0+\ker\widetilde{T}_1$, for which obviously we have
	$p_{\tilde{\mathrm{k}}}(\lV)=\ker \widetilde{T}_1$ and $\lV\subseteq\lW^+$. 
	It is also easy to check that $\lV^{[\perp]}=\lW_0+\ker T_1$ (see \cite[Corollary 3.2]{ES22}).
	For this $\lV$ we get $\widetilde{\lG}=\ker\widetilde{T}_1$
	and $\lG=\{0\}$, hence both $U$ and $U^*$ are zero operators. Since $\lG^{[\perp]}\cap\ker T_1=\ker T_1$
	and $\widetilde\lG^{[\perp]}\cap\ker \widetilde{T}_1=\{0\}$, it is easy to read 
	that the above expression for $\lV_U^{[\perp]}$ gives the right space $\lW_0+\ker T_1$.
\end{remark}

\subsection{Classification}

Now we are ready to formulate and prove the main result concerning a classification of realisations (of interest) of abstract Friedrichs operators.

\begin{theorem}\label{thm:classif}
Let $(T_0,\widetilde{T}_0)$ be a joint pair of abstract Friedrichs operators 
on $\lH$ and let $T$ be a closed realisation of $T_0$, i.e.~$T_0\subseteq T\subseteq T_1$. In what follows we use $\lV_U$ to denote the space \eqref{eq:V_U}
for a given $U$.
\begin{itemize}
\item[i)] $T$ is bijective if and only if there exists a bounded operator
$U:\ker\widetilde{T}_1\to\ker T_1$ such that $\dom T=\lV_U$.

\item[ii)] $\dom T\subseteq \lW^+$ if and only if there exist a closed subspace 
$\widetilde{\lG}\subseteq \ker\widetilde{T}_1$ and a non-expansive linear operator $U:(\widetilde{\lG},\iscp{\cdot}{\cdot}) \to (\ker T_1,-\iscp{\cdot}{\cdot})$, i.e.~the norm of $U$ with respect to the indicated norms is less than or equal to $1$, such that $\dom T=\lV_U$.

\item[iii)]  $T$ is a bijective realisation with signed boundary map 
if and only if there exists a non-expansive linear operator $U:(\ker\widetilde{T}_1,\iscp{\cdot}{\cdot}) \to (\ker T_1,-\iscp{\cdot}{\cdot})$ such that $\dom T=\lV_U$.

\item[iv)] $\dom T\subseteq \dom T^*$ if and only if there exist closed subspaces 
$\widetilde{\lG}\subseteq \ker\widetilde{T}_1$ and $\lG\subseteq\ker T_1$ and a unitary transformation
$U:(\widetilde{\lG},\iscp{\cdot}{\cdot}) \to (\lG,-\iscp{\cdot}{\cdot})$
(cf.~Remark \ref{rem:isom-char}) such that $\dom T=\lV_U$.

\item[v)] $\dom T= \dom T^*$ if and only if there exists a unitary transformation
$U:(\ker\widetilde{T}_1,\iscp{\cdot}{\cdot}) \to (\ker T_1,-\iscp{\cdot}{\cdot})$ such that $\dom T=\lV_U$.

\item[vi)] By the mapping $U\mapsto T_1|_{\lV_U}$, a one-to-one correspondence between realisations $T$, i.e.~$\dom T$, and classifying operators $U$ is established in each of the above cases.
\end{itemize}
\end{theorem}

\begin{proof}
Existence of such $U$ in all parts is a direct consequence of Lemma \ref{lm:isometry}. 
Thus, it remains to comment only the converse of each claim.

Since in all parts $U$ is bounded, by Remark \ref{rem:V_U_closed} we have that $\lV_U$ is 
a closed subspace of $\lW$ containing $\lW_0$. This means that 
$T_1|_{\lV_U}$ is indeed a closed realisation of $T_0$.

In part i) it is evident that $p_{\tilde{\mathrm{k}}}(\lV_U)=\ker\widetilde{T}_1$, hence we just 
apply Lemma \ref{lm:isometry}(i) to conclude.

For parts ii) and iii) we need to show that $\lV_U\subseteq \lW^+$ 
(bijectivity in part iii) is again consequence of Lemma \ref{lm:isometry}(i); see also Remark \ref{rem:Vcond}). For an arbitrary
$u=u_0+U\tilde\nu+\tilde\nu\in\lV_U$ we have
$$
\iscp{u}{u} = \iscp{U\tilde\nu}{U\tilde\nu} + \iscp{\tilde\nu}{\tilde\nu}
	\geq - \iscp{\tilde\nu}{\tilde\nu} + \iscp{\tilde\nu}{\tilde\nu} = 0 \,,
$$
where we have used that the norm of $U$ is less than or equal to 1
(see also Theorem \ref{thm:abstractFO-prop}(ix)). 

Let us recall that by Theorem \ref{thm:abstractFO-prop}(xi) we have 
$\dom (T_1|_{\lV_U})^*=\lV_U^{[\perp]}$. Thus, for parts iv) and v) we 
need to show that $\lV_U\subseteq \lV_U^{[\perp]}$ and (only for 
part v)) $\lV_U^{[\perp]}\subseteq\lV_U$. 
This can be done using Remark \ref{rem:V_Uperp}, but let us present here 
a direct proof. 
For arbitrary $u=u_0+U\tilde\nu+\tilde\nu$ and $v=v_0+U\tilde\mu+\tilde\mu$
from $\lV_U$, similarly as in the previous calculations, we have
$$
\iscp{u}{v} = \iscp{U\tilde\nu}{U\tilde\mu} + \iscp{\tilde\nu}{\tilde\mu}
	= - \iscp{\tilde\nu}{\tilde\mu} + \iscp{\tilde\nu}{\tilde\mu} = 0 \,,
$$
where we have used that $U$ is an isometry. 
Thus, $\lV_U\subseteq\lV_U^{[\perp]}$. 

Let us prove now the opposite inclusion for $U$ given in 
part v). Let $v\in \lV_U^{[\perp]}\subseteq\lW$. 
By Theorem \ref{thm:abstractFO-prop}(ix), there exist 
$v_0\in \lW_0$, $\mu \in \ker T_1$, $\tilde\mu\in \ker\widetilde{T}_1$, 
such that $v=v_0+\mu+\tilde\mu$. For any $u=u_0+U\tilde\nu +\tilde\nu\in \lV_U$, we have
\begin{align*}
0=\iscp{u}{v} &=\iscp{U\tilde\nu}{\mu}+\iscp{\tilde\nu}{\tilde\mu}\\
&= \iscp{U\tilde\nu}{\mu}-\iscp{U\tilde\nu}{U\tilde\mu} = \iscp{U\tilde\nu}{\mu-U\tilde\mu}\,,
\end{align*}
where we have used that $U$ is a unitary transformation.
The identity above holds for any $\tilde\nu\in \ker\widetilde{T}_1$. 
Since $U$ is surjective and $(\ker T_1, -\iscp{\cdot}{\cdot})$ is a Hilbert space, we get $\mu=U\tilde\mu$. Thus, $v\in \lV_U$ and hence, $V_U^{[\perp]}\subseteq \lV_U$.

Surjectivity of the map $U\mapsto T$ follows from parts i)-v), while 
injectivity holds by Lemma \ref{lm:isometry}(v).
\end{proof}

\begin{remark}
In \cite[Section 4]{AEM-2017} Grubb's classification was applied on 
abstract Friedrichs operators, which differs significantly from the
method of the previous theorem. For instance, in the result just developed 
a realisation is bijective if and only if the classifying operator is
defined on the whole kernel (Theorem \ref{thm:classif}(i)), 
while in the theory of \cite[Chapter 13]{Grubb}
(see also \cite[Theorem 17]{AEM-2017}) the same holds for bijective 
classifying operators. Another difference is that Grubb's classification 
is developed around the reference operator, while such distinguished
operator is not needed here. One can notice the same also in the 
symmetric case (see e.g.~\cite{CMO22}) when comparing von Neumann's (\emph{absolute}) theory
(cf.~\cite[Section 13.2]{Schmudgen}) and the (\emph{relative}) theory developed
by Kre\u\i n, Višik and Birman (cf.~\cite[Section 13.2]{Grubb}).

If we focus on bijective realisations with signed boundary map, then 
the result of part iii) of the previous theorem (see also part vi))
offers a full and explicit 
characterisation contrast to \cite[Theorem 18]{AEM-2017}, where 
the result is optimal only when kernels are isomorphic (cf.~Corollary 
\ref{cor:eqdom}).
\end{remark}

\begin{remark}
It is evident that parts iv) and v) of Theorem \ref{thm:classif} pertain to closed skew-symmetric and skew-selfadjoint realisations of $L_0$ (see \eqref{eq:TLSnotation}), respectively. 
Furthermore, in light of Remark \ref{rem:m-accretive}, we observe that parts ii) and iii) are associated to accretive and m-accretive realisations, respectively. 
Such results can readily be found in the literature for the case $S=\mathbbm{1}$, which are in complete agreement with ours (cf.~Remark \ref{rem:UinH}). However, the theory extends for even more general nonlinear relations, following both the von Neumann approach \cite{PT24}
and the more recent approach based on boundary systems (or boundary quadruples) 
\cite{ACE23, Trostorff23, WW20}. It is worth mentioning that boundary systems were introduced in \cite{SSVW15} as a generalisation of boundary triplets, with the advantage over the latter being the applicability of the theory irrespective of the values of the deficiency indices \cite{WW18}, which is the feature we also have with our results. 

In conclusion, the main novelty of Theorem \ref{thm:classif} lies in connecting the theory with abstract Friedrichs operators and considering general positive symmetric parts $S$, while still preserving the geometrical structure.   
\end{remark}

By Theorem \ref{thm:classif} we know that the number of certain type of 
realisations agrees with the number of corresponding classifying 
operators $U$.
For instance, it is easy to deduce the number of isomorphisms between 
Hilbert spaces. Hence, having this point of view at our hands, 
we can formulate the following
straightforward quantitative generalisation of Corollary \ref{cor:eqdom}.

\begin{corollary}\label{cor:eqdom-number}
Let $(T_0,\widetilde{T}_0)$ be a joint pair of abstract Friedrichs operators 
on $\lH$ and let us denote by $m$ the cardinality of the set of all subspaces $\lV$ of $\lW$ such that $\lV=\lV^{[\perp]}$, 
i.e.~such that $(T_1|_{\lV},\widetilde{T}_1|_{\lV})$ is a pair of mutually adjoint bijective realisations related to $(T_0,\widetilde{T}_0)$.
	\begin{itemize}
		\item[i)] If $\dim\ker T_1 \neq\dim\ker\widetilde T_1$, then $m=0$.
		
		\item[ii)] If $\dim\ker T_1 =\dim\ker\widetilde T_1=0$, then $m=1$.
		
		\item[iii)] If $\dim\ker T_1 =\dim\ker\widetilde T_1=1$, then $m=2$ in the real case, and $m=\infty$ in the complex case. 
		
		\item[iv)] If $\dim\ker T_1 =\dim\ker\widetilde T_1\geq 2$, then $m=\infty$.
	\end{itemize}
\end{corollary}

Let us conclude the section by illustrating the obtained results with several examples.

\begin{example}
\begin{itemize}
\item[a)] In Example \ref{exa:1Db} we have commented that $\dim\ker T_1=\dim\ker \widetilde{T}_1 =1$. Thus, since the problem was addressed in the real setting, by Corollary \ref{cor:eqdom-number} there are two closed subspaces with the property that $\lV=\lV^{[\perp]}$. They are precisely $\lV_{\alpha}$, $\alpha\in\{-1,1\}$ (see Example \ref{exa:1Da}).

\item[b)] Let us consider operators from examples \ref{exa:1Da} and \ref{exa:1Db} on $(0,\infty)$, instead of the bounded interval $(a,b)$, i.e.~$L_0u=u'$ and $\lH=L^2((0,\infty);\R)$. Then (see \cite[Example 3.2]{Schmudgen}) we have $d_+(L_0)=d_+(-iL_0)=1$ and $d_-(L_0)=d_-(-iL_0)=0$.
Thus, there is no closed subspace $\lV\subseteq\lW$ such that $\lV=\lV^{[\perp]}$, or in other words, there is no skew-selfadjoint realisation of the operator $-iL_0$.

\item[c)] The previous example is very specific since there is also only one bijective realisation. This can be justified by Theorem \ref{thm:classif}(i), but we also refer to \cite[Theorem 13]{AEM-2017}. 

Let us now present an example where still there is no closed subspace $\lV\subseteq\lW$ such that $\lV=\lV^{[\perp]}$, but for which there are infinitely many bijective realisations. More precisely, we need $\min\bigl\{\dim\ker T_1,\dim\ker\widetilde{T}_1\bigr\}\geq 1$ and $\dim\ker T_1 \neq\dim\ker\widetilde T_1$.

Let $\lH=L^2((0,1);\C^2)$ 
(all conclusions are also valid for the real case) and $\lD=C^\infty_c((0,1);\C^2)$.
For $\vu\in\lD$ and
$$
\mA(x):=
	\begin{bmatrix} 
	1 & 0  \\
	0 & 1-x
	\end{bmatrix}
$$
we define $T_0\vu := (\mA\vu)' + \vu$ and $\widetilde{T}_0\vu:=-(\mA\vu)'+\mA'\vu+\vu$.
It is easy to see that $(T_0,\widetilde{T}_0)$ is a joint pair of abstract Friedrichs operators (just apply Theorem \ref{thm:abstractFO_L0+S} or notice that $T_0$ is a classical Friedrichs operator \cite[Section 5]{EGC}).
As usual, we put $T_1:=\widetilde{T}_0^*$ and $\widetilde{T}_1:=T_0^*$.
Since both $T_1$ and $\widetilde{T}_1$ are of a block structure, calculations of the kernels can be done by studying each component separately. More precisely, $\vu=(u_1,u_2)\in\ker T_1$ if and only if
\begin{align*}
	u_1'+u_1=0 \quad \hbox{and} \quad (a_2u_2)'+u_2=0 \,,
\end{align*} 
where $a_2(x):=1-x$. Thus, we can apply the available results for scalar ordinary differential equations (see e.g.~the second example of \cite[Section 6]{ES22}).

Informally speaking, the equation above for the first component $u_1$ contributes with 1 for both $\dim\ker T_1$ and $\dim\ker \widetilde{T}_1$. 
On the other hand, the second equation contributes with 1 for $\dim\ker T_1$ and 0 for $\dim\ker\widetilde{T}_1$.
The overall result then reads
	\begin{align*}
	\dim \ker T_1=2 \quad \mathrm{and} \quad \dim\ker\widetilde T_1=1\;,
	\end{align*}
which corresponds to what we wanted to get.
\end{itemize}
\end{example}

\begin{remark}
	Another possibility to connect joint pairs of abstract Friedrichs operators $(T_0,\widetilde{T}_0)$ with symmetric operators is to study the operator matrix 
	$$
	J_0=\begin{bmatrix}0&T_0\\ \widetilde{T}_0&0\end{bmatrix}
	$$ 
	on $\lH\oplus\lH$, with $\dom J_0=\lD\oplus \lD$. 
	Indeed, $J_0$ is symmetric (cf.~\cite[Theorem 2.17]{JP17}) and 
	$$
	J_1:=J_0^*=\begin{bmatrix}0&T_1\\ \widetilde{T}_1&0\end{bmatrix} \,, \quad
		\dom J_1=\lW\oplus\lW \,.
	$$
	Moreover, $d_+(J_0)=d_-(J_0)$ (\cite[Theorem 2.20]{JP17}) and it is easy to see that all self-adjoint realisations of $J_0$ are given by $J_1|_{\lV^{[\perp]}\oplus\lV}$ (see Theorem \ref{thm:abstractFO-prop}(xi)), where $\lV\subseteq\lW$ is a closed subspace containing $\lW_0$ (or equivalently $\lD$).
\end{remark}

\section{Symmetric case}\label{sec:symm}

In this last part of the paper we focus on symmetric operators
and present several results that can be directly extracted from the 
theory just developed.

\begin{corollary}
Let $A$ be a densely defined symmetric operator on $\lH$ 
and let $S_1$, $S_2$ be bounded self-adjoint linear operators such that $S_2$ is in 
addition strictly positive. Define an indefinite inner product on $\dom A^*$ by
\begin{equation}\label{eq:iscpA}
\iscpA{u}{v} := i\Bigl(\scp{A^* u}{v} - \scp{u}{A^*v}\Bigr) \;, \quad u,v\in\dom A^* \;.
\end{equation}
Then we have the following.

\begin{itemize}
\item[i)] It holds
$$
\dim\ker (A^*-S_1-iS_2) = d_+(A) \quad \hbox{and} \quad	
	\dim\ker (A^*-S_1+iS_2)= d_-(A) \,,
$$
where $d_\pm(A)$ denote deficiency indices of $A$ (cf.~\cite[Section 3.1]{Schmudgen}).

\item[ii)] $\dom A^* = \dom \overline{A} \dotplus \ker (A^*-S_1-iS_2) \dotplus 
\ker (A^*-S_1+iS_2)$,
where the sums are direct and all spaces on the right-hand side 
are pairwise $\iscpA{\cdot}{\cdot}$-orthogonal. 

\item[iii)] There is one-to-one correspondence between all closed symmetric realisations 
of $A$ and all unitary transformations $U$ between any closed isomorphic subspaces 
of $(\ker (A^*-S_1+iS_2),\iscpA{\cdot}{\cdot})$ and $(\ker (A^*-S_1-iS_2),-\iscpA{\cdot}{\cdot})$,
respectively. 

\item[iv)] There is one-to-one correspondence between all self-adjoint realisations 
of $A$ and all unitary transformations $U:(\ker (A^*-S_1+iS_2),\iscpA{\cdot}{\cdot})\to
(\ker (A^*-S_1-iS_2),-\iscpA{\cdot}{\cdot})$.

\item[v)] Correspondences of parts iii) and iv) can be expressed by $U\mapsto A_U=A^*|_{\dom A_U}$, where 
\begin{equation*}
\dom A_U:=\bigl\{u_0+U\tilde\nu + \tilde\nu : u_0\in\dom \overline{A}, \, \tilde\nu\in \dom U\bigr\} \,,
\end{equation*}
and $A_U(u_0+U\tilde\nu + \tilde\nu)=\overline{A} u_0+ (S_1+iS_2)\tilde\nu + (S_1-iS_2)U\tilde\nu$.
\end{itemize}
\end{corollary}

\begin{proof}
If we define $T_0:=iA-iS_1+S_2$ and $\widetilde{T}_0:=-iA+iS_1+S_2$, 
then the pair $(T_0,\widetilde{T}_0)$ 
is a joint pair of abstract Friedirchs operators by Theorem \ref{thm:abstractFO_L0+S}. 
Moreover, corresponding indefinite inner product \eqref{eq:D} 
agrees with $\iscpA{\cdot}{\cdot}$ (see Remark \ref{rem:L0+C}).

Therefore, the statements of the corollary follow from 
Lemma \ref{lem:defect}, Theorem \ref{thm:abstractFO-prop}(ix) and 
Theorem \ref{thm:classif} (note that $\dom A_U$ agrees with \eqref{eq:V_U} for the above choice of 
$(T_0,\widetilde{T}_0)$).
%
%
\end{proof}

If $S_i=\alpha_i\mathbbm{1}$, $i=1,2$, where $\alpha_1\in\R$ and $\alpha_2>0$, 
then the statement of the previous theorem is well-known and can be found 
in many textbooks on unbounded linear operators. For instance, in \cite{Schmudgen}
part i) is present in Section 3.1, part ii) in Proposition 3.7 (von Neumann's formula)
and parts iii)-v) are studied in Section 13.2 as part of the von Neumann extension theory
(see also \cite[Chapter X]{Conway}). 
Moreover, the correspondence given in part v) completely agrees with the one of 
\cite[Theorem 13.9]{Schmudgen} since for this choice of bounded operators $S_i$, $i=1,2$,
the same $U$ represents a unitary transformation when the standard inner product of the 
Hilbert space $\lH$ is considered (see Remark \ref{rem:UinH}).

Let us just remark that the geometrical point of view provided in part ii), i.e.~orthogonality 
with respect to $\iscpA{\cdot}{\cdot}$, is something that is not commonly present,
although $\iscpA{\cdot}{\cdot}$ is (up to a multiplicative constant). 
More precisely, in \cite[Definition 3.4]{Schmudgen} (see also Lemma 3.5 there) 
the indefinite inner product $-i\iscpA{\cdot}{\cdot}$
is referred to as the \emph{boundary form} and it is an important part of 
the extension theory of boundary triplets (\cite[Chapter 14]{Schmudgen} and \cite[Section 13.4]{Grubb}; see also \cite{SSVW15} for more general boundary systems).
A more advanced study of boundary forms for Hilbert complexes can be found in a recent work \cite{HPS23}.

Of course, in the standard theory of symmetric operators it is usually satisfactory to observe only the case $S_1=0$ and $S_1=\mathbbm{1}$. Thus, the preceding corollary may seem like an excessive technical complication.
Here we want to stress one more time that our primary focus was in developing a classification result for abstract Friedrichs operators where such approach can be justified, e.g.~by perceiving that not all bijective realisations of $T_0=L_0+S$ correspond to skew-symmetric realisations of $L_0$ (see Section \ref{sec:classif}).
Therefore, our intention is to see the last corollary principally as a way to connect two theories, while an additional abstraction can sometimes offer a better sense of the underlying structure.

\section{Acknowledgements}
This work is supported by the Croatian Science Foundation under project 
IP-2022-10-7261 ADESO. 

The authors would like to express their gratitude to the anonymous referees for their careful work and for several
remarks that helped to improve the article.

\end{document}